\documentclass[a4paper,10pt,reqno]{amsart}

\usepackage{amssymb, amsmath, amsthm, amsfonts, comment}
\usepackage{graphicx, xypic}
\usepackage[colorlinks]{hyperref} % comment out for arXiv % why? -NA
\usepackage{float}
\usepackage{setspace} % required for neat, redefined \marginpar

\let\oldtocsection=\tocsection
\let\oldtocsubsection=\tocsubsection
\let\oldtocsubsubsection=\tocsubsubsection
\renewcommand{\tocsection}[3]{\hspace{0em}\oldtocsection{#1}{#2}{#3}}
\renewcommand{\tocsubsection}[3]{ \hspace{1em} \oldtocsubsection{#1}{\small{#2}}{\small{#3}} }
\renewcommand{\tocsubsubsection}[3]{\hspace{2em}\oldtocsubsubsection{#1}{\small{#2}}{\small{#3}}}

\newcommand{\al}[1]{\begin{align*}#1\end{align*}}

\newcommand{\qquotes}[1]{``{#1}''}

\newcommand{\set}[1]{\left\{{#1}\right\}}

\newcommand\xyhook{\ar@{^{(}->}}

\newcommand\isoto{\stackrel{\sim}{\To}}

\newcommand\into{\hookrightarrow}
\newcommand\To{\longrightarrow}

\newcommand\Hom{\operatorname{Hom}}
\renewcommand\hom{\mathcal{H}om }
\DeclareMathOperator\Ext{Ext}
\newcommand\RHom{\operatorname{RHom}}

\renewcommand\P{\mathbb P}
\newcommand\Gr{\operatorname{Gr}}
\newcommand\LGr{\operatorname{LGr}}
\newcommand{\Pf}{\mathbbm{Pf}}
\newcommand\pt{\operatorname{pt}}

\newcommand\rk{\operatorname{rank}}
\newcommand\GL{\operatorname{GL}}
\newcommand\SL{\operatorname{SL}}
\newcommand\Sp{\operatorname{Sp}}

\newcommand\C{\mathbb C}

\newcommand{\cA}{\mathcal{A}}

\newcommand{\cE}{\mathcal{E}}
\newcommand\cF{\mathcal F}
\newcommand{\cG}{\mathcal{G}}

\newcommand{\cO}{\mathcal{O}}
\newcommand{\cP}{\mathcal{P}}
\newcommand{\cS}{\mathcal{S}}
\newcommand{\cT}{\mathcal T}

\newcommand{\cX}{\mathcal{X}}

\newcommand{\Sym}{\operatorname{Sym}\!}
\newcommand{\Wedge}{\mbox{\scalebox{1.3}{$\wedge$}}}

\newcommand{\Crit}{\operatorname{Crit}}
\newcommand{\stack}[2]{\left[\, #1 \, \big/ \, #2\, \right]}

\newcommand\Br{\mathsf{BBr}}

\newcommand\dual{\vee}

\makeatletter \@addtoreset{equation}{section} \makeatother

\theoremstyle{plain}
\newtheorem{prop}[equation]{Proposition}
\newtheorem{thm}[equation]{Theorem}
\newtheorem{lem}[equation]{Lemma}

\newtheorem{keythm}{Theorem}

\theoremstyle{remark}
\newtheorem{rem}[equation]{Remark}

\theoremstyle{definition}

\newcommand{\pgap}{\vspace{7pt}}

\newcommand{\marginparstretch}{0.6}
\let\oldmarginpar\marginpar
\renewcommand\marginpar[1]{\-\oldmarginpar[\framebox{\setstretch{\marginparstretch}\begin{minipage}{\marginparwidth}{\raggedleft\scriptsize #1}\end{minipage}}]{\framebox{\setstretch{\marginparstretch}\begin{minipage}{\marginparwidth}{\raggedright\scriptsize #1}\end{minipage}}}}

\newcommand\I{\mathcal I}
\newcommand\PP[1]{\mathbb{P}^{\hspace{.8pt}#1}}

\newcommand{\Into}{\ensuremath{\lhook\joinrel\relbar\joinrel\rightarrow}}
\newcommand\Mapsto{\ensuremath{\shortmid\joinrel\relbar\joinrel\rightarrow}}
\newcommand\Langle{\big\langle}
\newcommand\Rangle{\big\rangle}
\renewcommand\Gr{\operatorname{Gr}}
\renewcommand\Pf{\operatorname{Pf}}

\renewcommand\pt{\operatorname{pt}}
\renewcommand\Hom{\operatorname{Hom}}

\renewcommand\Ext{\operatorname{Ext}}
\newcommand\ext{\mathcal{E}xt}

\newcommand\beq[1]{\begin{equation}\label{#1}}
\newcommand\eeq{\end{equation}}

\title{Quintic threefolds and Fano elevenfolds}
\author[E. Segal and R. P. Thomas]{Ed Segal and Richard P. Thomas}

\begin{document}

\begin{abstract} \noindent
The derived category of coherent sheaves on a general quintic threefold is a central object in mirror symmetry. We show that it can be embedded into the derived category of a certain Fano elevenfold. Our proof also generates related examples in different dimensions.
\end{abstract}
\maketitle
\tableofcontents

%%%%%%%%%%%%%%%%%%%%%%%%%%%%%%%%%%%%%%%%%%%%%%%%%%%%%%%%%%%%%%%%%%%%%%%%%%%

\section{Introduction}\label{sec:intro}

Fix a 10-dimensional vector space $V\cong\C^{10}$. Consider the Grassmannian
\beq{Gr}
\Gr:=\Gr(2,V)\ \subset\ \P\big(\Wedge^2V\big)
\eeq
and the Pfaffian variety
\beq{Pf}
\Pf:=\Pf_{10}=\big\{[\omega]\in\P (\Wedge^2V^\dual):\omega^{\wedge 5}=0\in\Wedge^{10}V^\dual\big\}\ \subset\ \P\big(\Wedge^2V^\dual\big).
\eeq
Notice that \eqref{Pf} is a quintic hypersurface in $\PP{44}$, singular in codimension 5. It is the classical projective dual of \eqref{Gr}. Now pick a 5-dimensional subspace
\beq{Udef}
\C^5 \cong U\subset\Wedge^2 V^\dual
\eeq
or equivalently a 40-dimensional subspace
$$
\C^{40}\cong U^\perp \subset \Wedge^2 V.
$$

We intersect \eqref{Gr} with $\P (U^\perp)\cong\PP{39}$ and \eqref{Pf} with $\P(U)\cong\PP{4}$. This defines an 11-dimensional linear section of the Grassmannian
\beq{Y1}
Y_1:=\P (U^\perp)\cap\Gr
\eeq
and a quintic 3-fold
\beq{Y2}
Y_2:=\P (U)\cap\Pf
\eeq
respectively.
For a generic choice of $U$, both $Y_1$ and $Y_2$ are smooth. Conversely, Beauville \cite[Proposition 8.9]{Beau} shows that the \emph{general} smooth quintic threefold $Y_2\subset\PP{4}$ arises in this way.\footnote{Though not uniquely. Different presentations of a given $Y_2$ as linear sections of $\Pf$ give rise to different dual Fano elevenfolds $Y_1$.} 

The moduli space of the Fanos $Y_1$ is a generically-finite cover of the moduli space of the quintics $Y_2$; in particular
$$
h^1(T_{Y_1})=101=h^1(T_{Y_2}).
$$
Moreover their cohomologies are as closely related as possible. By the Lefschetz hyperplane theorem, in degrees less than the middle, the cohomologies of $Y_1$ and $Y_2$ are the same as those of their ambient spaces $\Gr$ and $\PP{5}$ respectively. The same is true in degrees higher than the middle after a shift by twice the codimension. Finally in the middle degree, the nonzero pieces of the cohomologies have the same dimensions:
$$ h^{3,0}\ \ h^{2,1}\ \ h^{1,2}\ \ h^{0,3} \quad=\quad 1\ \ 101\ \ 101\ \ 1 $$
for $H^3(Y_2)$ and
$$ h^{7,4}\ \ h^{6,5}\ \ h^{5,6}\ \ h^{4,7} \quad=\quad 1\ \ 101\ \ 101\ \ 1 $$
for $H^{1\hspace{-.5pt}1}(Y_1)$. 
Our main result categorifies this relation.

\begin{keythm} \label{main}
There is a full and faithful embedding $D^b(Y_2)\into D^b(Y_1)$.
\end{keythm}

In fact this is a special case of a more general result, Theorem \ref{maingeneral} below, which also covers some other interesting examples.

Theorem \ref{main} should have various consequences when combined with mirror symmetry. In particular, the Fukaya categories of $Y_1$ and $Y_2$ should also be related after a rescaling of the Novikov parameter $q$, with the latter a summand of the former. Taking Hochschild cohomologies, we should find that the quantum cohomology ring $QH^*(Y_2)$ of $Y_2$ should be a summand of $QH^*(Y_1)$ after applying a rescaling of the quantum parameter $q$. Setting $q=0$ would recover the embedding of the Hodge diamond of $Y_2$ into that of $Y_1$ alluded to above. \pgap

Theorem \ref{main} would follow directly from Kuznetsov's beautiful work on homological projective duality \cite{HPD, KuGr} if one could prove \cite[Conjecture 5]{KuGr} for Gr(2,10). In short, Kuznetsov conjectures that \eqref{Gr} and \eqref{Pf} should be homologically projectively dual varieties \cite{HPD} once one replaces Pf with an appropriate categorical crepant resolution of its singularities (which has so far only been found in lower dimensions). This would imply a relation between the derived categories of the linear sections of Gr and of the orthogonal linear sections of Pf. In particular, for $\P(U)$ chosen to avoid the singularities of Pf, we would find that $D^b(Y_1)$ has a semi-orthogonal decomposition
\beq{YX}
D^b(Y_1)=\Langle\cA,\cA(1),\ldots,\cA(4),D^b(Y_2)\Rangle,
\eeq
where $\cA$ is the category generated by the exceptional collection
$$
\set{\Sym^3S,\,\Sym^2S,\,S,\,\cO} \quad\mathrm{on}\ Y_1
$$
and $S$ is the (restriction to $Y_1$ of the) universal subbundle on Gr.
That is,
\al{
D^b(Y_2)\,&\cong{}^\perp\big\langle \Sym^i S (j),\ 0\le i\le3,\ 0\le j\le 4\big\rangle \\
&=\Big\{E\in D^b(Y):\RHom_X(E,\Sym^i S(j))=0\text{ for all }0\le i\le3,0\le j\le 4\Big\}.
}
We expect furthermore that, up to some mutations and twists by line bundles, the inclusion $D^b(Y_1)\into D^b(Y_2)$ should be given by the Fourier-Mukai kernel $\I_\Gamma$, the ideal sheaf of
$$
\Gamma:=\Big\{(\phi,P)\in\Pf\times\Gr\,:\,\ker\phi\cap P\ne0\Big\}\subset Y_1\times Y_2.
$$
Here $\ker\phi\subset V$ denotes the kernel of $\phi\in\Wedge^2V^\dual$ when thought of as a (skew) linear map $V\to V^\dual$. The correspondence $\Gamma$ associates $\phi$ to the locus of 2-planes $P\subset V$ which intersect $\ker\phi$ nontrivially.
\pgap

Since we deliberately avoid the singularities of $\Pf$ the methods of \cite{HPD, KuzECs, KuGr} are surely strong enough to prove Theorem \ref{main} without finding the right categorical resolution of singularities of $\Pf$. Here however we take a different approach, inspired by string theory. 

In their paper \cite{HoriTong}, Hori and Tong wrote down a non-abelian gauged linear sigma model (GLSM) that gave a physical explanation of the so-called `Pfaffian-Grassmannian' derived equivalence between two particular Calabi-Yau threefolds. The paper \cite{ADS}
%(by Addington, Donovan, and the first-named author)
gave a mathematical treatment of Hori and Tong's construction at the level of B-brane categories. 

In this paper we take the techniques and results of \cite{ADS} and apply them to a slightly more general GLSM. This gives us a more general result, Theorem \ref{maingeneral}, which says that we have a derived embedding between certain smooth linear sections of the Pfaffian variety and the dual smooth linear sections of a Grassmannian. Special cases then give the quintic threefold case of Theorem \ref{main}, the Pfaffian-Grassmannian equivalence, and examples with $K3$ surfaces and Calabi-Yau 5-folds.

Although our terminology is different, our approach is intimately connected with homological projective duality; see \cite{Betal} for another situation in which HPD is realised via GLSMs.

This paper is based heavily on \cite{ADS}; the only new technical ingredient is the work in Section \ref{quadbundles} to show the vanishing of a Brauer class. Consequently we have made little attempt to make this paper self-contained, and we refer the reader to \cite{ADS} for background, motivation, references and more detailed explanations.

\smallskip\noindent\textbf{Acknowledgements.} We would like to thank Nick Addington and Will Donovan for allowing us to re-use the arguments of the paper \cite{ADS}. We also thank Nick Addington for useful conversations and for computational help, and Ivan Smith for generous help with the Fukaya category and quantum cohomology. 

Sometime after posting this paper we became aware of the paper \cite{IM} in which our embedding of derived categories was explicitly conjectured.

R.T. was partially supported by EPSRC programme grant EP/G06170X/1.

\section{Geometric setup and statement of theorem}

Fix vector spaces $V, U$ and $S$ of dimensions $n,\,k\leq{ n \choose 2}$ and $2$ respectively, and consider
$$ \cX =\stack{\big(\!\Hom(S, V) \oplus(U\otimes\Wedge^2 S)\big)}{\GL(S)}. $$
The square brackets indicate that we consider this as an Artin stack (rather than a scheme-theoretic or GIT quotient).
We let $x$ and $p$ denote elements of $\Hom(S,V)$ and $U\otimes \Wedge^2 S$ respectively. We have open substacks
\begin{eqnarray*}
&&\iota_1\colon X_1 = \set{ \rk x =2 }\Into \cX, \\
&&\iota_2\colon X_2 = \set{ p \neq 0} \Into \cX.
\end{eqnarray*}
The locus $X_1$ is a variety: the total space of the vector bundle 
$$ \cO(-1)\otimes U \To \Gr(2, V).$$
The locus $X_2$ is still an Artin stack; it is a bundle over $\P(U)$ whose fibres are the `linear' Artin stacks 
\beq{fibreF}
\stack{ \Hom(S, V)}{\SL(S)}.
\eeq
We can rephrase this: we let $\cP$ be the stack 
\beq{cPdef}
\cP =  \stack{ (U\otimes\Wedge^2 S)\setminus \set{0}}{\GL(S)},
\eeq
which is a Zariski-locally trivial bundle of stacks over $\P(U)$, with fibre $B\!\SL_2$. Then $X_2$ is a vector bundle over $\cP$, whose fibres are the vector spaces $\Hom(S, V)$.

The loci $X_1$ and $X_2$ are the semi-stable loci for a positive or negative GIT stability condition, so one of the GIT quotients is $X_1$, and the other is the underlying scheme of $X_2$. 

Now fix a surjective linear map\footnote{The dual $A^\dual\colon U\to\Wedge^2V^\dual$ will later specialise to the injection \eqref{Udef} of the Introduction.}
\beq{A} A\colon \Wedge^2 V \To U^\dual.\eeq
This defines an (invariant) function $W\colon\cX\to \C$ by
\beq{W1}
W =  p \circ A \circ \wedge^2 x.
\eeq
We also use $W$ to denote the restriction of this function to $X_1$ and $X_2$. Finally we fix a $\C^*$ action (an \qquotes{R-charge}) on $\cX$ by giving $x$ weight zero and $p$ weight 2, so that both $X_1$ and $X_2$ are invariant and $W$ has weight 2. Given this data, the three pairs
$$(\cX, W),\quad (X_1, W) \quad \mbox{and} \quad (X_2, W)$$
are all Landau-Ginzburg B-models, as defined in \cite{Ed}, and we have restriction functors 
\beq{restr}
D^b(X_1, W) \stackrel{\iota_1^*}{\longleftarrow} D^b(\cX, W) \stackrel{\iota_2^*}{\To} D^b(X_2, W)
\eeq
between their categories of (global) matrix factorizations.

\begin{rem}\label{meaningofMF} For a Landau-Ginzburg B-model $(X,W)$, the  objects of $D^b(X, W)$ are curved dg-sheaves $(\cE, d_\cE)$, where $\cE$ is a quasi-coherent sheaf on $X$. Strictly speaking we reserve the word \emph{matrix factorization} for the special case in which $\cE$ is a finite-rank vector bundle, but (by definition) every object in $D^b(X,W)$ is equivalent to a matrix factorization. This is discussed in detail in \cite[\S2]{ADS}.
\end{rem}

Now let
$$Y_1 \subset \Gr(2,V)$$
be the zero locus of the section
\beq{sigma}
\sigma:= A \circ \wedge^2 x \,\in\, \Gamma\,\big(\!\Gr(2, V), \cO(1)\otimes U^\dual \big).
\eeq
$Y_1$ is closely related to the critical locus $\Crit(W)$ of the function $W$ on $X_1$. In fact, a point $(x, p)\in X_1$ lies in $\Crit(W)$ iff $\sigma(x)=0$, and $p$ lies in the kernel of the linear map
$$d\sigma : U\otimes \Wedge^2 S \To T_{[x]}^\vee\Gr(2,V) $$
which measures the derivatives of $\sigma$ in the fibre directions (and which is well-defined when $\sigma(x)=0$). Hence we always have $Y_1\subset \Crit(W)$, and we have $Y_1=\Crit(W)$ if and only if the section $\sigma$ is transverse to the zero section,  \emph{i.e.}~if and only if $Y_1$ is  a smooth codimension-$k$ complete intersection. From now on we restrict to generic $A$ for which this is true.

By global Kn\"orrer periodicity \cite[Theorem 3.4]{Shipman}, there is a canonical equivalence
\beq{knorrer}D^b(Y_1)  \isoto D^b(X_1, W). \eeq

This describes the left hand side of \eqref{restr}. The right hand side is more complicated. Over $\P(U)$ we have a family of 2-forms on $V$ up to scale given to us by $A^\dual$ \eqref{A}. The locus where these have rank $<n-1$ is a variety
$$Y_2\,\subset\,\P(U),$$
the intersection of the Pfaffian variety $\Pf\subset \P(\Wedge^2 V^\dual)$ of degenerate two-forms on $V$ with the linear subspace
$$A^\dual\colon \P(U)\Into \P(\Wedge^2 V^\dual). $$
It follows that $Y_2$ is also the locus where the (degenerate) quadratic form $W$ on the fibres \eqref{fibreF} of $X_2\to\P(U)$ drops rank. In this situation there is a more complicated version of Kn\"orrer periodicity; see Sections \ref{quadbundles}, \ref{pfaffianside} and \cite{ADS}. There is also a corresponding Brauer class, but we show this vanishes in Section \ref{quadbundles}.

The singular locus of $\Pf$ is a subvariety of codimension 6 inside  $\P(\Wedge^2 V^\dual)$ when $n$ is even, and codimension 10 when $n$ is odd. Therefore if $n$ is even and $k\leq 6$, or $n$ is odd and $k\leq 10$, then for a generic choice of $A$ the variety $Y_2$ is smooth. (If $k$ is larger than these bounds then $Y_2$ will never be smooth.) Under these assumptions, we prove in Section  \ref{pfaffianside} that $D^b(Y_2)$ embeds into a certain subcategory of $D^b(X_2, W)$.

Note that the variety $Y_1$ is Fano if $k< n$, Calabi-Yau if $k=n$, and general type if $k>n$. The canonical bundle of $Y_2$ is easy to calculate if $n$ is even: $Y_2$ is Fano for $k>n/2$, Calabi-Yau for $k=n/2$ and general type for $k<n/2$. When $n$ is odd the calculation is a little harder, but the three cases occur when $k>n$, $k=n$ and $k<n$ respectively.

\begin{thm}\label{maingeneral} Suppose that 
\begin{itemize}
\item[(i)] $k\le\min(n,10)$ if $n$ is odd, or
\item[(ii)] $k\le\min(n/2,6)$ if $n$ is even.
\end{itemize}
Assume also that $A$ is generic, so that both $Y_1$ and $Y_2$ are smooth. Then we have an admissible embedding
$$D^b(Y_2) \Into D^b(Y_1).$$
\end{thm}
Here \emph{admissible} means that the embedding admits a right adjoint.\footnote{By Serre duality the existence of a left adjoint is equivalent to the existence of a right adjoint.} It follows that $D^b(Y_1)$ has a semi-orthogonal decomposition whose last term is $D^b(Y_2)$; c.f.~\eqref{YX}.  See Remark \ref{orthogonal} for some discussion of the orthogonal complement to our embedding.

Setting $n=10,\,k=5$ gives Theorem \ref{main} of the Introduction. The case $n=k=7$ is the `Pfaffian-Grassmannian' equivalence, which is the subject of \cite{ADS}.

Setting $n=8,\,k=4$ gives an embedding of the derived category of a Pfaffian quartic K3 into the derived category of a codimension-4 linear section of $\Gr(2,8)$. Note that the \emph{general} quartic in $\PP{3}$ is Pfaffian \cite[Prop.~7.6]{Beau}.

Setting $n=k=9$ we get a novel derived equivalence between Calabi-Yau 5-folds.

\section{Grade-restriction windows}

Recall that $n= \dim V$, and let us set  $L=\tfrac{n-1}{2}$ for $n$ odd and $L=\tfrac{n}{2}$ for $n$ even.  Let $\cS$ be the following set of representations of $\GL(S)$:
$$ \cS = \set{ \Sym^l S \otimes (\det S)^m\ :\ l\in \big[0, L \big), \; m\in [0, n) }$$
if $n$ is odd, or
$$\cS = \set{ \Sym^l S \otimes (\det S)^m\,:\,l\in \big[0, L-2\big], \; m\in [0, n)  \,\mathrm{\ or\ }\, l=L-1,\; m\in\big[0,\tfrac n2\big)}
$$
if $n$ is even. Each representation induces a vector bundle on $\cX$ which we denote by the same letters. 

On restriction to $\Gr(2, V)\subset X_1$ we get a set of vector bundles which is very nearly the full strong exceptional collection found by Kuznetsov in \cite{KuzECs}, we have just replaced all bundles by their duals. So by Serre duality, our set is also a full strong exceptional collection. His collection is Lefschetz; ours is `dual Lefschetz' in that the blocks get bigger rather than smaller as one twists by $\cO(1)$. In HPD this would have the effect of swapping left and right in all resulting semi-orthogonal decompositions.

 The set $\cS$ is adapted to the `Grassmannian side' of our set-up; for the `Pfaffian side' we consider the set
\beq{T}\cT = \set{ \Sym^l S\otimes (\det S)^m\ :\ l\in [0, L), \; m\in [0, k) },\eeq
where $k=\dim U$ as before.  Notice that
\beq{numcond} 
\cT\subset\cS \mathrm{\ \ if\ and\ only\ if\ }k\leq n \mbox{ for $n$ odd, or }k\leq \tfrac{n}{2} \mbox{ for $n$ even.}
\eeq
(There is also a `reverse' numerical condition that implies that $\cS\subset \cT$, but this is less useful to us.) We let
$$\cG_1 = \left\langle \cS  \right\rangle \mbox{ and } \cG_2 = \left\langle \cT  \right\rangle \ \subset D^b(\cX)
$$
be the subcategories of $D^b(\cX)$ generated by $\cS$ and $\cT$, \emph{i.e.} the closures of $\cS$ and $\cT$ under mapping cones and shifts (but not direct summands). We also let
$$\cG_1^W \mbox{ and } \cG_2^W \ \subset D^b(\cX, W)$$
be the subcategories consisting of objects that are (homotopy-equivalent to) matrix factorizations whose underlying vector bundles are direct sums of shifts of the bundles appearing in $\cS$ and $\cT$ respectively.

\begin{prop}\label{windows} The restriction functors
$$\iota_1^*: \cG_1 \To D^b(X_1) \quad\mbox{and}\quad \iota_1^*: \cG_1^W \To D^b(X_1, W)$$
are both equivalences, and the restriction functors
$$\iota_2^*:\cG_2 \To D^b(X_2) \quad\mbox{and}\quad  \iota_2^*: \cG_2^W \To D^b(X_2, W)$$
are both embeddings.
\end{prop}

\begin{proof} The statements without $W$ are proved by exactly the same argument as for \cite[Proposition 4.1]{ADS}; restriction to $X_1$ or $X_2$ does not create any higher Ext groups between the respective sets of vector bundles, and the restriction of $\cS$ generates $D^b(X_1)$.\footnote{The vector bundles appearing in \cite{ADS} are powers of $S^\vee$ rather than $S$, but this makes no difference to the arguments.} The only additional ingredient we need is a minor generalisation of \cite[Lemma 4.5]{ADS} to cover the case when $n$ is even, and we provide this as Lemma \ref{Extvanishing} below. 

The statements with $W$ follow from the statements without $W$, using the proof of \cite[Proposition 4.9]{ADS} verbatim.
\end{proof}

\begin{lem}\label{Extvanishing} Let $n=\dim V$ be even, and let $\Sym^l S(-m)$ and $\Sym^{l'}S(-m')$ be two  $GL(S)$-representations lying in the set $\cS$. Then for any $t\geq 0$, we have the following vanishing of higher Ext groups 
$$\Ext^{>0}_{\Gr(2, V)}\big( \Sym^l S(-m), \, \Sym^{l'}S(t-m')\big)=0 $$
between the corresponding vector bundles on $\Gr(2,V)$.
\end{lem}
\begin{proof}
We follow the proof of \cite[Lemma 4.5]{ADS} quite closely. We need to prove the vanishing of higher cohomology for the bundle $\Sym^l S^\vee \otimes \Sym^{l'} S^\vee(m-m'-l'+t)$, and by decomposing this tensor product and arguing recursively, it's sufficient to deal with the summand $\Sym^{l'+l} S^\vee(m-m'-l'+t)$. We do this using the Borel--Weil--Bott algorithm, applied to the weight 
$\alpha= (\alpha_1, \alpha_2, 0,...,0)$ where
$$\alpha_1 = m-m'+l+t \hspace{1cm}\mbox{and}\hspace{1cm} \alpha_2= m-m'-l'+t. $$
Adding $\rho$ (half the sum of the positive roots) gives 
$$\alpha+\rho = (\alpha_1+n,\, \alpha_2+n-1, \, n-2,\, ...,\, 1 ).$$
Since $\alpha_1-\alpha_2=l'+l\geq 0$, the weight is dominant iff $\alpha_2\geq 0$, in which case the higher cohomology of the bundle vanishes. Furthermore, if either
$$ \alpha_2\in [2-n, -1] \hspace{1cm}\mbox{or}\hspace{1cm} \alpha_1\in [1-n, -2]$$
then two entries of $\alpha+\rho$ are the same, and all cohomology of the bundle vanishes. Now consider the definition of $\cS$. We must have either:
\begin{enumerate}
\item $l'=\tfrac12 n-1$ and $m'\leq \tfrac12 n -1$. Then $\alpha_2\geq 2-n$. Or
\item $l'\leq \tfrac12 n-2$. Then $\alpha_1-\alpha_2=l+l'\leq n-3$. Since $\alpha_1\geq -m'\geq 1-n$, we must either have $\alpha_1\in [1-n, -2]$ or $\alpha_2\geq -1-l-l' \geq 2-n$.\qedhere
\end{enumerate}
\end{proof}
\pgap

Under the numerical condition in \eqref{numcond} we have that
\beq{inclusions} \cG_2\subset \cG_1\quad\quad\mbox{and}\quad\quad \cG_2^W\subset\cG_1^W.\eeq

\begin{prop}\label{windowadjoint}
The restriction functors
$$\iota_2^*: \cG_1\To D^b(X_2) \quad\quad\mbox{and}\quad\quad  \iota_2^* : \cG_1^W \To D^b(X_2, W)$$
land in the subcategories $\iota_2^*(\cG_2)$ and $\iota_2^*(\cG_2^W)$ respectively. Given the numerical condition in \eqref{numcond}, these functors are the right adjoints to the inclusions \eqref{inclusions}.
\end{prop}
\begin{proof}
Consider first the statements without $W$. On $\cX$ replace $\cO_{\set{p=0}}$ by its Koszul resolution. The resulting sheaves $\Wedge^*(U\otimes\Wedge^2S)^\dual$ are all sums of line bundles. Restricting to $X_2$ the complex becomes acyclic, giving the corresponding relation in $D^b(X_2)$. Repeatedly applying this relation, and its twist by line-bundles $\cO(i)$, shows that the line bundle $\cO(m)$ on $X_2$ lies in $\iota_2^*(\cG_2)$, for any value of $m$. Similarly, tensoring the acyclic complex with $\Sym^l S$ (and applying this relation repeatedly) shows that $\Sym^l S(m)$ lies in $\iota_2^*(\cG_2)$ for any $m$, provided that $l<L$. Therefore $\iota_2^*(\cG_1)\subset \iota_2^*(\cG_2)$.

The statement that $\iota_2^*$ is the right adjoint to the inclusion $\cG_2\subset \cG_1$ can be checked on the generators, so we need to know that  if $E\in \cT$ and $F\in \cS$ then 
$$\RHom_{\cX}(E, F) = \RHom_{X_2}(\iota_2^* E, \iota_2^* F).$$
In other words, for these two bundles the $\Ext^0$ group doesn't change upon restriction to $X_2$, and on $X_2$ there are no higher Ext groups. This is proved by precisely the same method as in the proof of Proposition \ref{windows}.

As before, the statements with $W$ follow from the statements without $W$ by the techniques of \cite[Proposition 4.9]{ADS}.
\end{proof}

We define 
$$\Br(X_2, W) \subset D^b(X_2, W)$$
to be the image of $\cG_2^W$; it is hopefully the category of B-branes in some associated SQFT.

If we assume the numerical condition from \eqref{numcond}, then putting together Kn\"orrer periodicity \eqref{knorrer} with Propositions \ref{windows} and \ref{windowadjoint} shows that \eqref{inclusions} gives an embedding
\beq{embedding}\Br(X_2, W) \Into D^b(Y_1)\eeq
as a right-admissible subcategory. To prove Theorem \ref{maingeneral}, it remains to show that $D^b(Y_2)$ embeds as a right-admissible subcategory of $\Br(X_2,W)$.

\begin{rem}\label{orthogonal} To better compare our result with HPD (see the Introduction) let us make some remarks about right-orthogonals.

Purely formally, the right-orthogonal to $\Br(X_2, W)$ inside $D^b(Y_1)$ is the same thing as the kernel of the right adjoint functor $D^b(Y_1)\to \Br(X_2, W)$. By Proposition \ref{windowadjoint} this adjoint functor is given by: apply Kn\"orrer periodicity to get to $D^b(X_1, W)$, extend into $\cG_1^W$, then apply $\iota_2^*$. 

Under Kn\"orrer periodicity, the structure sheaf $\cO_{Y_1}$ maps to the skyscraper sheaf along the locus $X_1|_{Y_1} \subset X_1$. If we view the section $\sigma$ \eqref{sigma} as  section of a bundle on $X_1$ (by pulling up), then its zero locus is $X_1|_{Y_1}$.  Hence $\cO_{X_1|_{Y_1}}$ has a Koszul resolution using $\sigma$, and this can be perturbed to produce an equivalent matrix factorization
\beq{koszul}\xymatrix@C=20pt{
\cO(-k)\; \ar[r]<2pt>^-{\sigma}
&  \;\;... \;\;\ar[l]<2pt>^-{p}   \ar[r]<2pt>^-{\sigma}
&\; U(-1)\; \ar[r]<2pt>^-{\sigma} \ar[l]<2pt>^-{p}
& \;\cO. \ar[l]<2pt>^-{p}
}\eeq
 Kn\"orrer periodicity commutes with tensoring by sheaves on $Y_1$, so the vector bundle $\Sym^l S (-m)$ on $Y_1$ maps to the matrix factorization \eqref{koszul}, tensored with $\Sym^l S (-m)$. Suppose we choose $l$ and $m$ such that every summand in this matrix factorization corresponds to a representation lying in the set $\cS$. Then the lift of this object into $\cG_1^W$ is obvious -- we take exactly the same expression, but consider it as a matrix factorization on $\cX$.

Now use the symmetry between $\sigma$ and $p$, and regard this matrix factorization as a pertubation of (a twist and shift of) the skyscraper sheaf along the locus $\set{p=0}\subset \cX$. Evidently this object lies in the kernel of the functor $\iota_2^*$.

This argument shows the following statement: take a representation $\Sym^l S \otimes (\det S)^m$ such that $\Sym^l S \otimes (\det S)^{m+t}$ lies in $\cS$ for every $t\in [0,k]$. Then the associated vector bundle on $Y_1$ lies in the right-orthogonal to $\Br(X_2, W)$, and so (by the results of Section \ref{pfaffianside}) it lies in the right-orthogonal to $D^b(Y_2)$. 

This shows that that the right-orthogonal to $D^b(Y_2)$ includes all the vector bundles lying in a certain 'rectangle' within $\cS$. For example, setting $n=10$ and $k=5$ we see that (after a twist by $\cO(-4)$) the 20 vector bundles from \eqref{YX} are indeed orthogonal to $D^b(Y_2)$. To agree with the HPD story, we should really prove that these bundles generate the entire right-orthogonal to $D^b(Y_2)$. In particular this would mean that $D^b(Y_2)$ is actually equivalent to $\Br(X_2, W)$.

\end{rem}

\section{Quadratic bundles arising from symplectic bundles}\label{quadbundles}

Given a vector bundle equipped with an everywhere non-degenerate quadratic form, Kn\"orrer periodicity implies that the category of matrix factorizations on the total space of the bundle is equivalent to the derived category of the base space, once we twist the latter by a Brauer class. If the bundle admits a maximally-isotropic subbundle $M$ then the Brauer class vanishes, and the skyscraper sheaf $\cO_M$ can be used to construct an equivalence between the matrix factorization category and the ordinary derived category of the base. However, the existence of $M$ is a rather stronger condition than the vanishing of the Brauer class.

In this section we describe, for quadratic vector bundles of a particular type,  an alternative construction which proves the vanishing of the Brauer class and provides the equivalence between the two categories.

\subsection{Cleanly intersecting submanifolds of $\set{W=0}$}\label{cleanlyintersecting}

Before discussing any quadratic vector bundles we make a rather general observation. Let $(X, W)$ be any Landau-Ginzburg B-model, and let
$$A, B\subset\big\{W=0\big\}\subset X$$
be submanifolds of the zero locus of $W$. Assume that $A$ and $B$ intersect cleanly, so $A\cap B$ is also a submanifold, and we have an excess normal bundle 
$$ E \,=\, \frac{T_X}{T_A + T_B} \ \ \mathrm{on\ }A\cap B. $$
Let $r$ denote the rank of $E$, and let $a$ be the codimension of $A\subset X$. Then in the ordinary derived category $D^b(X)$, a standard computation with the Koszul resolution gives the Ext sheaves between $\cO_A$ and $\cO_B$ as
%$$\ext^{\bullet}(\cO_A, \cO_B) = \Wedge^{\bullet}\big(E^\dual[1]\big) \otimes \det N_{A/X} [-\codim A]. $$
\beq{sheafy}
\ext^i(\cO_A, \cO_B) = \Wedge^{a-i}E^\dual \otimes \det N_{A/X}, \quad a-r\le i\le a,
\eeq
and zero otherwise.

Since $A$ and $B$ lie in $\set{W=0}$, the sheaves $\cO_A$ and $\cO_B$ define objects in $D^b(X, W)$. By a minor extension of the argument in  \cite[\S A.4]{ASS}, there is a spectral sequence computing the local sheaf of morphisms between them in $D^b(X,W)$, whose 2nd page consists of the sheaves \eqref{sheafy} and whose differential is given by wedging with the section\footnote{This section is well-defined, since $W$ vanishes along $A$ and $B$.}
$$dW\colon \cO_{A\cap B} \To E^\dual.$$
Suppose that this section of $E^\dual$ is transverse to 0 with zero locus $Z$ (which is therefore a component of the critical locus of $W$). Then by the 3rd page only one term remains:
$$
\ext^{a-r}_{D^b(X, W)}(\cO_A, \cO_B) = \cO_Z\otimes\det E^\dual\otimes\det N_{A/X}.
$$
%\beq{homformula}\ext^\bullet_{D^b(X, W)}(\cO_A, \cO_B) = \cO_Z \otimes K_{A\cap B}\otimes K_B^{-1}[\dim A\cap B - \dim B] \eeq
Thus the spectral sequence collapses to give
\beq{homformula}
R\hom_{D^b(X, W)}(\cO_A, \cO_B)\,=\,\cO_Z \otimes K_{A\cap B}\otimes K_B^{-1}[\dim A\cap B-\dim B].
\eeq
Here $K_{A\cap B}$ and $K_B$ denote the canonical bundles, and we have used $a-r=\dim B-\dim A\cap B$.

\subsection{Another version of Kn\"orrer periodicity}

Let $S$ and $V$ be two symplectic vector spaces. Let $\theta_S\in\Wedge^2S$ be the Poisson bivector on $S$ and $\Omega_V$ be the symplectic form on $V$. Then the vector space $\Hom(S, V) $ carries a non-degenerate quadratic form
\beq{W2}
W\colon x \Mapsto \Langle\Omega_V,\wedge^2x\,(\theta_S)\Rangle.
\eeq
By Kn\"orrer periodicity, the category $D^b(\Hom(S,V),W)$ is equivalent to the derived category of a point $D^b(\pt)$, non-canonically. An equivalence is specified by any exceptional object that generates the category. One option is to choose a Lagrangian $L\subset V$ and take the skyscraper sheaf of the corresponding maximally-isotropic subspace:
\beq{LM}
M:= \Hom(S, L) \subset \Hom(S,V).
\eeq
This gives an equivalence 
$$R\Hom(\cO_M,\ \cdot\ )\colon D^b\big(\!\Hom(S,V), W\big) \isoto D^b(\pt) $$
sending $\cO_M$ to $\cO_{\pt}$. Our next result says that in this situation there is a more canonical generator, independent of any choices, and hence equivariant with respect to both $\Sp(S)$ and $\Sp(V)$.

Let $\LGr(S)$ denote the Lagrangian Grassmannian of $S$, and let 
$$ \LGr(S) \stackrel{\pi_1}{\longleftarrow}\LGr(S) \times \Hom(S,V) \stackrel{\pi_2}{\longrightarrow} \Hom(S,V) $$
denote the projections onto the two factors. The vector bundle $\pi_1$ carries a family of non-degenerate quadratic forms $\pi_2^* W$ and a natural maximally-isotropic subbundle
\beq{Ndef}
N := \Hom(S/\Lambda, V),
\eeq
where $\Lambda\to\LGr(S)$ is the tautological Lagrangian subbundle of $S$. The skyscraper sheaf $\cO_N$ is an object of $D^b\big(\!\LGr(S)\times \Hom(S,V), \pi_2^* W\big)$.

\begin{prop}\label{choicefreegenerator} The object
\beq{cEdef}
\cE := R\pi_{2*} \!\left(\cO_N\otimes (\det \Lambda)^{-\frac12\dim V}\right)
\,\in\,D^b\big(\!\Hom(S,V), W\big)
\eeq
is exceptional and generates the category.
\end{prop}
\begin{proof}
Choose a Lagrangian $L\subset V$, giving a maximally-isotropic subspace $M \subset \Hom(S,V)$ as in \eqref{LM}. Let $\widetilde{M} = \LGr(S)\times M$ be the corresponding maximally-isotropic subbundle of $\LGr(S)\times \Hom(S, V)$. The functor
$$R\pi_{1*} R\hom(\cO_{\widetilde{M}\,},\ \cdot\ )\colon D^b\big(\!\LGr(S) \times \Hom(S,V), \pi_2^* W\big) \To D^b(\LGr(S)) $$
is an equivalence by the simplest version of Kn\"orrer periodicity in families. Moreover the square
\beq{square}\xymatrix@C=90pt{
 D^b\big(\!\LGr(S) \times \Hom(S,V), \pi_2^* W\big)  
\ar[r]^(.6){R\pi_{1*}R\hom(\cO_{\widetilde{M}}, \ \cdot\ )}  \ar[d]_{R\pi_{2*}}
     &    D^b(\LGr(S)) \ar[d]^{R\Gamma} \\
D^b(\Hom(S,V), W) \ar[r]^(.6){R\Hom(\cO_M,\ \cdot\ )} &  D^b(\pt) 
}
\eeq
commutes by the projection formula.

Now we take our tautological maximal isotropic subbundle $N$ \eqref{Ndef} and compute 
$$R\hom\big(\cO_{\widetilde{M}}, \cO^{}_N\big)$$
in $D^b\big(\!\LGr(S) \times \Hom(S,V), \pi_2^* W\big)$, using the analysis from Section \ref{cleanlyintersecting}. The submanifolds $\widetilde{M}$ and $N$ intersect cleanly 
along the subbundle
$$\widetilde{M}\cap N = \Hom (S/  \Lambda, L)$$
with excess normal bundle
$$ E = \Hom(  \Lambda , V/L) $$
over $\widetilde{M}\cap N$. From the definition of $W$, we have that
$$dW\big|_x  =\theta_S\circ x^\dual\circ\Omega_V \in \Hom(V,S)=\Hom(S, V)^\vee, $$
where we consider $\theta_S$ and $\Omega_V$ as skew elements of $\Hom(S^\dual,S)$ and $\Hom(V,V^\dual)$ respectively. Therefore at a point $(\Lambda,x)$ of $\widetilde{M}\cap N$ the derivative of $\pi_2^* W$ is the map
$$ \theta_S\circ x^\dual\circ\Omega_V \in \Hom(L, S/ \Lambda)=\Hom(V/L^\dual,  \Lambda^\vee), $$
 where the last isomorphism follows from the Lagrangian property of $L$ and $\Lambda$. This lies in the fibre of $E^\dual$ over $(\Lambda,x)$. The resulting section of $E^\dual$ has zero locus $\set{x=0} = \LGr(S)$, so it is transverse to the zero section and we may apply \eqref{homformula}. 

By an elementary calculation
$$K_{\widetilde{M}\cap N}\otimes K_N^{-1} = (\det  \Lambda)^{\frac12\dim V}\otimes(\det L)^{-\frac12\dim S} $$
so \eqref{homformula} gives
$$ R\hom\big(\cO_{\widetilde{M}}, \cO^{}_N\big) = \cO_{\LGr(S)} \otimes (\det  \Lambda)^{\frac12\dim V}\otimes(\det L)^{-\frac12\dim S}\big[\!-\tfrac{1}{4}\dim S\cdot\dim V\big]. $$
Consequently, the upper arrow in the square \eqref{square} takes $ \cO_N\otimes (\det  \Lambda)^{-\frac12\dim V} $ to a shift of $\cO_{\LGr(S)}$. Therefore going the other way round the square shows that $\cE$ is taken by the lower arrow to the same shift of
$\cO_{\pt}$. In particular, $\cE$ is isomorphic to a shift of $\cO_M$ in the category $D^b\big(\!\Hom(S,V), W\big)$.
\end{proof}

\begin{rem}\label{rank1locus}
Section 5.4 of \cite{ADS} is similarly concerned with finding a canonical exceptional generator of the category $D^b(\Hom(S,V), W)$, in the particular case that $\dim S=2$ and $\dim V = 4$. There it was proved (by another method) that the skyscraper sheaf on the locus $\set{\rk x \leq 1}$ is such a generator. We now explain how that object relates to the construction given here.

The proof of Proposition \ref{choicefreegenerator} gives us a second canonical exceptional generator
$$
\cE' = R\pi_{2*} \!\left(\cO_N\otimes (\det \Lambda)^{-\frac12\dim V}\otimes \pi_1^*K_{\LGr(S)}\right).
$$
This is easy to see:  since the upper arrow in the square \eqref{square} takes $\cO_N\otimes (\det \Lambda)^{-\frac12\dim V}$ to a shift of $\cO_{\LGr(S)}$, if we throw in this twist we instead obtain a shift of $K_{\LGr(S)}$, and then applying $R\Gamma$ still gives us a shift of $\cO_{pt}$.

Now set $\dim S =2$. Then $\Lambda\to\LGr(S)$ is just $\cO(-1)\to\P (S)$. The image of the map $\pi_2|_N$ is exactly the rank $\leq1$ locus in $\Hom(S,V)$, so both $\cE$ and $\cE'$ are supported on this locus. If we further set $\dim V=4$ then two line-bundles cancel and we have:
$$ \cE' = R\pi_{2*\,}\cO_N\big(\tfrac12 \dim V -2\big)=R\pi_{2*}\cO_N=\cO_{\{\rk x\le 1\}}.$$
\end{rem}
\pgap

Since $\cE$ is canonical, Proposition \ref{choicefreegenerator} works in families. Let $S$ and $V$ be symplectic vector bundles over a base $B$, or even vector bundles carrying symplectic forms only up to scale.\footnote{By this we mean a section of $\Wedge^2V^\dual\otimes L$ for some line bundle $L$, such that the induced map $V\to V^\dual\otimes L$ is an isomorphism.} Then $\Hom(S,V)\stackrel{p}{\to}B$ carries a fibrewise non-degenerate quadratic form $W$ up to scale, given by the formula \eqref{W2}. We form the bundle
$$\pi_2\colon \LGr(S)\times^{}_B \Hom(S,V) \To \Hom(S,V)$$
carrying its tautological subbundle $\Lambda\subset S$. With this we can define
the maximally-isotropic subbundle $N:=\Hom(S/\Lambda,V)$ of $\LGr(S)\times_B \Hom(S,V)\to\LGr(S)$, and
\beq{EBdef}
\cE_B := R\pi_{2*}\Big(\cO_N \otimes (\det  \Lambda)^{-\tfrac12 \rk V}\Big).
\eeq
This is the global analogue of the object \eqref{cEdef}. Zariski-locally, there is also a version of the object $\cO_M$ of \eqref{LM}. The symplectic group is special, so  $V$ is Zariski-locally trivial as a symplectic bundle. Therefore, replacing $B$ by an open subset we may assume that the symplectic form on $V$ is constant. Hence it admits a trivial Lagrangian subbundle $L$, defining a maximally-isotropic subbundle $M\subset\Hom(S,V)$ by the formula \eqref{LM}. 

The proof of Proposition \ref{choicefreegenerator} now applies verbatim: $Rp_*R\hom(\cO_M,\cE_B)$ is a shift of a line bundle on our shrunken $B$. Since $D^b\big(\!\Hom(S,V),W\big)$ is generated over $D^b(B)$ by $\cO_M$, this shows that $\cE_B$ and $\cO_M$ are isomorphic up to a shift and a twist by a line bundle. That is, once we shrink $B$ to ensure that $V$ is trivial, we get the following isomorphism in $D^b\big(\!\Hom(S,V),W\big)$:
\beq{isom}
\cE_B\ \cong\ \cO_M\,\otimes\,(\det L)^{-\frac12\rk S}\big[-\tfrac14\rk S\cdot\rk V\big].
\eeq

It follows that for any $B$ the two Fourier-Mukai functors
\beq{equivalence}
\xymatrix@C=80pt{D^b(B) \ar@^{->}[r]<2pt>^-{\cE_B\otimes p^*(\ \cdot\ )}
& D^b\big(\!\Hom(S,V),W\big) \ar@^{->}[l]<1pt>^-{Rp_*R\hom(\cE_B,\ \cdot\ )}}
\eeq
are mutual inverses; \emph{i.e.} the natural adjunction map from their composition to the structure sheaf of the diagonal is a quasi-isomorphism. Again this can be checked locally, where it follows from \eqref{isom} and the corresponding result for $\cO_M$. Thus \eqref{equivalence} gives an equivalence with \emph{zero Brauer class}.

\section{Pfaffian side}\label{pfaffianside}

In this section we construct an embedding of $D^b(Y_2)$ into $\Br(X_2, W)$. The method is the one employed in \cite[\S 5]{ADS} supplemented with the construction from Section \ref{quadbundles} to produce a global Fourier-Mukai kernel.

We let $\pi$ denote the composition of the projections
$$X_2 \To \cP\To\P(U)$$
of \eqref{cPdef}. The first map is a vector bundle over $\cP$ with fibre $\Hom(S,V)$; the second is a bundle of stacks $B\!\SL_2$. The map $A^\dual\colon U\to\Wedge^2V^\dual$ of \eqref{A} defines a section of $\Wedge^2V^\dual(1)$ over $\P(U)$ -- \emph{i.e.} a family of 2-forms on the $n$-dimensional vector space $V$, defined up to scale. The variety $Y_2\subset\P(U)$ is the locus where this family drops in rank, either from $n$ to $n-2$ (if $n$ is even) or from $n-1$ to $n-3$ (if $n$ is odd).

Let $K\to Y_2$ be the kernel of the family of 2-forms:
$$
0\To K\To V\stackrel{A}{\To}V^\dual(1).
$$
It is a subbundle of the trivial bundle $V\times Y_2$ of rank 2 or 3. Dividing out by $K$ gives a quotient bundle
$$q\colon V\To\widehat{V}:= V / K.$$
 The family of forms $A$ descends to give a family $\widehat A\in\Gamma
\big(\Wedge^2\widehat V^\dual(1)\big)$ of symplectic forms (up to scale) on the vector bundle $\widehat V\to Y_2$.

Now consider the vector bundle
$$\Hom\!\big(S, \widehat{V}\big) $$
on the stack $\cP|_{Y_2}$. Since $S$ is 2-dimensional, this carries an associated family of non-degenerate quadratic forms (up to scale) given by the formula \eqref{W2}. Via $q$ this pulls back to a family of degenerate quadratic forms on the bundle $\Hom(S, V) = X_2|_{Y_2}$; this is precisely the restriction of the function $W$ \eqref{W1}.

We now apply the method of Section \ref{quadbundles} to the symplectic bundles $S,\widehat V$ over the base $B=\cP|_{Y_2}$ to give a object
$$
\cE\in D^b\big(\!\Hom\!\big(S, \widehat{V}\big), W\big)
$$
by the formula \eqref{EBdef}. Pulling up to $\Hom(S, V)$ and pushing forward into $X_2$ gives an object
$$ j_*q^*\cE \in D^b(X_2, W), $$
where $j: X_2|_{Y_2}\into X_2$ denotes the inclusion map.
%We claim that this object $j_*q^*\cE$ `looks like' the structure sheaf on $Y_2$, \emph{i.e.}~
We claim that
\beq{quasiisom}
\cO_{Y_2}\stackrel{\mathrm{id}}{\To}R\pi_*R\hom^{\ }_{D^b(X_2,W)}\!\left(
j_*q^*\cE,\,j_*q^*\cE\right)
\eeq
is a quasi-isomorphism. Again, we can check this locally on $Y_2$.
 
We proceed as at the end of Section \ref{quadbundles}.
Even though our base $B=\cP|_{Y_2}$ is a stack rather than a scheme, the bundle $\widehat V$ is pulled back from the scheme $Y_2$. Therefore we can use the same Zariski-locally-trivial argument. We replace $\P(U)$ by an open subset, thus shrinking $X_2$ and $Y_2$ by basechange. We may then assume $\widehat V$ is trivial and pick a trivial Lagrangian subbundle $L\subset\widehat V$. This defines a maximal isotropic subbundle $M\subset\Hom(S,\widehat V)$ by the formula \eqref{LM}, and we get the isomorphism \eqref{isom}. That is $\cE$ is isomorphic to $\cO_M$ up to a shift and a twist by a line bundle. In particular (now that we have shrunk $X_2$ and $Y_2$ to produce an $M$) we get an isomorphism between
$$
j_*q^*\cE \quad\mathrm{and}\quad j_*\cO_{q^{-1}(M)}
$$
in $D^b(X_2, W)$ up to a shift and a twist.
% \footnote{We emphasise again that we have replaced $Y_2$ by an open subset, and so basechanged $X_2\to Y_2$ accordingly. $M$ need not exist globally.}

A key result of \cite[Proposition 5.3 and  Remark 5.13]{ADS} was that when such an $M$ exists we have that
$$
\cO_{Y_2}\stackrel{\mathrm{id}}{\To}
R\pi_*R\hom^{\ }_{D(X_2,W)}\!\left(j_*\cO_{q^{-1}M},\,j_*\cO_{q^{-1}M} \right)
$$
is a quasi-isomorphism. Therefore \eqref{quasiisom} is also a quasi-isomorphism over our open set, and hence also globally.
% 
% Let $\cU\subset \P(U)$ be a Zariski-open set, and let $Y'= Y_2\cap \cU$. Assume that $\cU$ is small enough that the symplectic bundle $\widehat{V}$ becomes trivial over $Y'$. Then within the open set $Y'$ we may find a Lagrangian subbundle $L$ of $\widehat{V}$, and a corresponding maximally-isotropic subbundle $M\subset \Hom(S, \widehat{V})$. Furthermore the isomorphism \eqref{isom} holds, \emph{i.e.}~$\cE$ is  isomorphic to $\cO_M$ up to a shift and a twist by a line bundle. Consequently we have an isomorphism between the two objects
% $$
% j_*q^*\cE \quad\mathrm{and}\quad j_*\cO_{q^{-1}(M)}\quad \in D^b(X_2|_{\cU}, W)
% $$
% up to a shift and a twist. One of the key calculations of \cite[Proposition 5.3 and Remark 5.13]{ADS} was the result that
% $$j_*\cO_{Y'} \stackrel{\mathrm{id}}{\To} R\pi_*R\hom^{\ }_{D^b(X_2|_\cU,W)}\!\left(j_*\cO_{q^{-1}M},\,j_*\cO_{q^{-1}M} \right)$$
% is a quasi-isomorphism. Hence \eqref{quasiisom} must be a quasi-isomorphism locally in $Y_2$, and so also globally.

Using $j_*q^*\cE$ as a Fourier-Mukai kernel, we consider the functor
\begin{eqnarray*}
F\colon D^b(Y_2) \!&\To&\! D^b(X_2, W), \\
\cF \!&\Mapsto&\! j_*\big(\pi^*\cF\otimes q^*\cE\big).
\end{eqnarray*}

\begin{rem} Just as in \cite{ADS} there is a technical issue here - applying this kernel produces curved complexes of coherent sheaves, but it is not \emph{a priori} obvious that they are always equivalent to matrix factorizations (see Remark \ref{meaningofMF}). So we must justify the claim that this functor really does land in $D^b(X_2, W)$ rather than some larger category. Since $X_2$ is smooth this should probably follow from some foundational theorem, but here it is solved by Proposition \ref{FFF} below, which is a much stronger statement.
\end{rem}

The functor $F$ has a right adjoint
$$F^R\colon \cG \Mapsto R\pi_*R\hom(j_*q^*\cE, \cG)$$
and the equation \eqref{quasiisom} says exactly that $F^R\circ F$ is the identity. Therefore $F$ embeds $D^b(Y_2)$ as a right-admissible subcategory of $D^b(X_2, W)$.

 To conclude the proof of Theorem \ref{maingeneral} we need only show the following.

\begin{prop} \label{FFF} The image of the functor $F$ is contained in the subcategory
$$\Br(X_2, W) \subset D^b(X_2, W). $$
\end{prop}
\begin{proof}
Recall that $\cE$ is obtained by pushing down the sheaf $\cO_N(\tfrac12 \rk \widehat{V})$. 
If we replace $\cO_N$ here by its Kozsul resolution, we can get a free resolution
of $\cE$ by bundles of the form 
$$\Wedge^a \widehat{V}^\dual\otimes \Sym^b S \otimes (\det S)^c; $$
see for example \cite[\S A2.6]{Eisenbud}. Furthermore, the symmetric powers of $S$ that occur lie in the range $b\le\tfrac12\rk \widehat{V}$. This is precisely the range of symmetric powers included in our set $\cT$ \eqref{T}, since $\rk \widehat{V}$ is either $n-3$ (if $n$ is odd) or $n-2$ (if $n$ is even).
The remainder of the argument is exactly the same as for \cite[Proposition 5.9]{ADS}.
\end{proof}

As mentioned in Remark \ref{orthogonal}, we believe that $F$ is actually an equivalence between $D^b(Y_2)$ and $\Br(X_2, W)$.

\bigskip \noindent {\tt{richard.thomas@imperial.ac.uk \\ edward.segal04@imperial.ac.uk}} \medskip

\noindent Department of Mathematics \\
\noindent Imperial College London\\
\noindent London SW7 2AZ \\
\noindent United Kingdom


\begin{thebibliography}{9}

\bibitem[ADS]{ADS}
\textit{N.~Addington}, \textit{W.~Donovan} and \textit{E.~Segal}, 
The Pfaffian-Grassmannian equivalence revisited, Alg. Geom. \textbf{2} (2015), no. 3, 332--364.   \href{http://arxiv.org/abs/1401.3661}{\textsf{arXiv:1401.3661}}.

\bibitem[ASS]{ASS}
\textit{N.~Addington}, \textit{E.~Segal} and \textit{E.~Sharpe},
D-brane probes, branched double covers, and noncommutative resolutions, Adv. Theor. Math. Phys. \textbf{18} (2014) no. 6, 1369--1436. 
\href{http://arxiv.org/abs/1211.2446}{\textsf{arXiv:1211.2446}}.

\bibitem[B$^+$]{Betal}
\textit{M.~Ballard}, \textit{D.~Deliu}, \textit{D.~Favero}, \textit{M.~U.~Isik} and \textit{L.~Katzarkov},
 Homological Projective Duality via Variation of Geometric Invariant Theory Quotients, \href{http://arxiv.org/abs/1306.3957}{\textsf{arXiv:1306.3957}}.

\bibitem[Be]{Beau} \textit{A.~Beauville},
 Determinantal hypersurfaces,
 Michigan Math. J. \textbf{48} (2000), 39--64.

\bibitem[Ei]{Eisenbud}
\textit{D.~Eisenbud},
Commutative algebra with a view toward algebraic geometry,
 Graduate Texts in Mathematics \textbf{150}. Springer, 1994.

\bibitem[HT]{HoriTong}
\textit{K.~Hori} and \textit{D. Tong}, 
Aspects of non-abelian gauge dynamics in two-dimensional  $\mathcal{N}=(2,2)$ theories,
 J. High Energy Phys. \textbf{05} (2007) 079. 
 \href{http://arxiv.org/abs/hep-th/0609032}{\textsf{arXiv:hep-th/0609032}}.

\bibitem[IM]{IM}
\textit{A. Iliev} and \textit{L. Manivel},
Fano manifolds of Calabi-Yau type,
J. Pure Appl. Algebra \textbf{219} (2015), 2225--2244.
 \href{http://arxiv.org/abs/1102.3623}{\textsf{arXiv:1102.3623}}

\bibitem[K1]{HPD}
\textit{A. Kuznetsov},
Homological projective duality,
 Pub. Math. I.H.E.S. \textbf{105} (2007), 157--220.
 \href{http://arxiv.org/abs/math/0507292}{\textsf{math.AG/0507292}}.

\bibitem[K2]{KuzECs}
\textit{A. Kuznetsov},
Exceptional collections for Grassmannians of isotropic lines,
 Proc. London Math. Soc., \textbf{97}  (2008), 155--182.
 \href{http://arxiv.org/abs/math/0512013}{\textsf{math.AG/0512013}}.

\bibitem[K3]{KuGr}
\textit{A.~Kuznetsov},
Homological projective duality for Grassmannians of lines,
 \href{http://arxiv.org/abs/math/0610957}{\textsf{math.AG/0610957}}.

\bibitem[Se]{Ed}
\textit{E. Segal},
Equivalences between GIT quotients of Landau-Ginzburg B-models,
Comm. Math. Phys. \textbf{304} (2011), 411--432.
 \href{http://arxiv.org/abs/0910.5534}{\textsf{arXiv:0910.5534}}.

\bibitem[Sh]{Shipman}
\textit{I. Shipman},
A geometric approach to Orlov's theorem,
 Compos. Math. \textbf{148} (2012), 1365--1389.
 \href{http://arxiv.org/abs/1012.5282}{\textsf{arXiv:1012.5282}}.





\end{thebibliography}
\end{document}